\documentclass{article}

\usepackage{arxiv}

\usepackage[utf8]{inputenc} 
\usepackage[T1]{fontenc}    
\usepackage{hyperref}       
\usepackage{url}            
\usepackage{booktabs}       
\usepackage{amsfonts}       
\usepackage{nicefrac}       
\usepackage{microtype}      
\usepackage{lipsum}
\usepackage{mathtools}
\usepackage{float}
\usepackage{xcolor}
\usepackage{amssymb}
\usepackage{mathtools}
\usepackage{cite}
\usepackage[ruled,vlined]{algorithm2e}
\usepackage{subcaption}
\usepackage{indentfirst}

\title{The Regional Boundary Reconstruction Problem 
of the Initial State for Fractional Semilinear Systems\thanks{This is a preprint of a paper 
published in 'Int. J. Appl. Comput. Math. 12 (2026), no.~2, Paper No.~37, 19~pp.' at 
[\url{https://doi.org/10.1007/s40819-026-02104-y}].}}

\author{
Khalid Zguaid\\
EMA Team, LRST Laboratory\\
The Higher School of Education and Training\\
Ibnou Zohr University\\
Agadir, Morocco \\
\texttt{k.zguaid@uiz.ac.ma} \\
\And
Fatima Zahrae El Alaoui\\
TSAN Team, TSI Laboratory\\
Faculty of Sciences\\
Moulay Ismail University\\
11201, Meknes, Morocco \\
\texttt{f.elalaoui@umi.ac.ma} \\
\And
Delfim F. M. Torres \\
Center for Research and Development in Mathematics and Applications (CIDMA),\\
Department of Mathematics,\\ University of Aveiro,\\ 3810-193 Aveiro, Portugal \\
\texttt{delfim@ua.pt} \\ 
}

\newtheorem{proposition}{\quad Proposition}[section]
\newtheorem{theorem}{\quad Theorem}[section] 
\newtheorem{definition}[theorem]{\quad Definition}

\newtheorem{remark}{\quad Remark }
\newenvironment{proof}[1][Proof]{\underline{\textbf{#1}} }{ \hfill \rule{0.5em}{0.5em}}

\begin{document}

\maketitle

\begin{abstract}
Observability is a fundamental concept in control theory. Its primary purpose is to look 
into whether it is possible to reconstruct the system's initial state using only the 
information from its outputs. This paper focuses on the regional reconstruction problem 
of the initial state for a semilinear time-fractional system, which refers to the possibility 
of recovering the value of the initial state on a desired boundary sub-region instead 
of the whole evolution domain or its boundary. To achieve this objective, we use the 
analytical method, where we suppose that the system's dynamic generates an analytic semigroup. 
First, by establishing an internal sub-region, we establish a connection between the concepts 
of regional observability and regional boundary observability; we will later explain how to 
define the internal sub-region. Then, by imposing suitable assumptions on the analytic semigroup 
and the system's non-linearity, we give the main theorems of this research from which we deduce 
a sequence that converges to the unknown initial state on the desired boundary sub-region. 
Moreover, we present an algorithm that produces some numerical simulations which align 
closely with our theoretical findings.  
\end{abstract}



\section{Introduction}

Fractional Calculus, as a mathematical discipline, explores the realms of differentiation 
and integration beyond the confines of integer orders, delving into the intricacies of 
non-integer order differentials and integrals; this is a classical field that initially attracted 
limited applied interest because it was only a set of complicated mathematical formulas with no clear 
physical interpretation or real-world applications. This was the case until three decades ago when 
fractional calculus started appearing in various application domains. Many problems modeled by classical, 
or integer order, systems, which could not catch and describe many behaviors of those problems, 
are now being modeled by fractional order systems. The main reason is that the classical derivative 
and integral are local operators and do not consider the history or past states of the studied phenomenon, 
which does not adequately reflect many real phenomena since most phenomena depend on their past states. 
This is not the case for fractional, or non-integer, order operators, due to their hereditary and 
non-locality properties which allow all past states and experiences of the phenomenon to be taken 
into account in the current state of the considered modeling system. From the many applications 
of fractional calculus, we mention: energy harvesting in dynamical systems \cite{app2}, 
bio-medicine \cite{app5}, unconfined groundwater \cite{app3}, viscoelasticity \cite{app4} 
and solid mechanics \cite{app1}, etc. In all of these works and others, it has been proven 
that the fractional model has a much better performance than the classical one.

Recent years have witnessed a growing interest in the use of fractional-order models 
to describe complex dynamical phenomena characterized by memory and hereditary effects. 
In particular, fractional calculus has found important applications in viscoelasticity, 
structural dynamics, and seismic engineering, where classical integer-order models often 
fail to accurately reproduce experimental observations. For instance, spline collocation 
methods combined with fractional models have been successfully employed for the seismic 
analysis of multiple degree-of-freedom systems equipped with viscoelastic dampers, 
showing improved accuracy and stability when compared with classical approaches \cite{sup.1}. 
Similar techniques have also been applied to the numerical modeling of viscoelastic dampers 
in structural buildings, further demonstrating the effectiveness of fractional models 
in capturing realistic dissipative behaviors in engineering structures \cite{sup.2}.

From a theoretical standpoint, significant progress has been made in the understanding 
and unification of fractional derivatives. In particular, a general framework encompassing 
linear fractional derivatives with power-function convolution kernels has been proposed, 
providing a clear interpretation of their interpolation properties and clarifying the 
analytical roots of widely used operators such as the Caputo derivative \cite{sup.3}. 
These results offer a solid mathematical foundation for the analysis of fractional 
evolution equations and justify the widespread use of the Caputo derivative in control 
and observability problems, especially due to its natural compatibility 
with classical initial conditions.

Despite these advances, most existing works dealing with fractional-order models primarily 
focus on numerical simulations or direct applications, while comparatively fewer contributions 
address inverse problems such as state reconstruction and observability, especially for 
infinite-dimensional and semilinear systems. In particular, the problem of reconstructing 
the initial state from partial or boundary measurements in semilinear fractional systems 
remains largely unexplored. This observation further motivates the present work, which aims 
to bridge this gap by combining recent developments in fractional calculus with regional 
boundary observability techniques from control theory.

Control theory, which is the main theme of this paper, is one of the areas of applied 
mathematics in which fractional calculus has been successfully applied. In \cite{appC}, 
we find many fractional calculus applications in this domain. Control theory encompasses 
many useful concepts, such as stability, controllability, and observability. Our primary 
interest lies in the notion of observability, which involves the pursuit of reconstructing 
the initial state of the system under consideration. In \cite{kalman}, Kalman has introduced, 
for the first time, observability for finite dimensional systems. For infinite dimensional 
systems, we refer to this concept sometimes as global observability, and it has been widely 
developed for various kinds of systems \cite{curtain,weiss}.  Additionally, it is extended 
to a broader case known as regional observability \cite{reg,me.2023.chap,me.2023.ajc,me.2023.ijdyc}, 
where the objective is always to identify and restore the initial state but only in a specific 
sub-region of the entire evolution domain, which is a more practical notion since it costs less 
to reconstruct the initial state regionally than globally and also since it is not always possible 
to reconstruct the initial state in the whole evolution domain. Regional observability was developed 
to the case where the desired sub-region is a part of the boundary (we refer to this as regional 
boundary observability), see \cite{bound1,bound2}. This concept finds its roots in actual problems 
faced in reality, specifically in the determination of energy exchanged in a casting plasma 
on a plane target perpendicular to the flow direction, based on measurements conducted using 
thermocouples, see \cite{bound1}. It was first investigated for linear systems \cite{bound2}; 
afterwards, it was extended to cover semilinear systems, see \cite{boun.semi1} 
for parabolic systems and \cite{boun.semi2} for hyperbolic ones. Regional observability 
and regional boundary observability were also treated for fractional systems, 
see \cite{RegAnal,me.bound.lin.2022.2} for time-fractional linear systems 
and \cite{me.semi.2021,me.semi.lin.2023} for semilinear ones.

Although regional observability and regional boundary observability have been extensively 
studied for integer-order systems, as well as for linear fractional systems, existing results 
for semilinear time-fractional systems remain limited and largely focused on internal 
reconstruction problems. In particular, the boundary reconstruction of the initial state 
for semilinear fractional systems has not been thoroughly investigated, mainly due to the 
intrinsic difficulties associated with boundary measurements, non-linearity, and the nonlocal 
nature of fractional derivatives. Moreover, direct reconstruction on the boundary often leads 
to instability and lack of robustness. These limitations motivate the need for a new framework 
capable of overcoming such difficulties by combining internal observability concepts 
with boundary reconstruction objectives. The present paper addresses this gap by proposing 
an indirect yet rigorous approach that links boundary observability to internal regional 
observability, thereby enabling a stable and effective reconstruction of the initial 
state on a prescribed boundary sub-region.

This work addresses the problem of reconstructing the initial state of a semilinear 
time-fractional system, governed by a Caputo derivative, on a prescribed boundary 
sub-region. The main novelty of this paper lies in extending the concept of regional 
boundary observability to semilinear time-fractional systems, a setting that has not 
been fully explored in the existing literature. A key contribution consists in 
establishing a rigorous theoretical link between regional boundary observability 
on a boundary sub-region and approximate regional observability of the linear part 
of the system on a suitably chosen internal sub-region whose boundary contains the 
region of interest. Based on this connection, we develop a reconstruction strategy 
relying on the Hilbert Uniqueness Method (HUM), leading to an explicit iterative 
algorithm for recovering the initial state in the internal sub-region and, consequently, 
on the desired boundary portion. The effectiveness and robustness of the proposed approach 
are validated through numerical simulations involving both zonal and pointwise sensors, 
illustrating the practical relevance of the theoretical results.

The present manuscript is constructed as follows: Section two states 
some preliminary notions that cover the needed notions to understand this paper better. 
In the third section, we present the considered system and the related results regarding 
the regional boundary observability, on $\Sigma$, of that system, in particular we show 
how to link the regional boundary observability of the considered system and the 
regional observability of its linear part, in $\omega$. In the same section, 
we use some results of regional observability of the considered semilinear system 
with the goal of reconstructing the initial state in $\omega$, and we give an 
algorithm that allows that reconstruction numerically. 
The fourth section is dedicated to demonstrating the efficacy of the proposed algorithm 
through the presentation of two numerical examples. Each example showcases a different 
type of sensor (zonal and pointwise), highlighting the algorithm's effectiveness.    


\section{Preliminary Tools}

Let $\Omega$ be an open and bounded set included in $\mathbb{R}^n$, with $n\geq 2$, and $\partial\Omega$ its boundary. 
Let us fix $]0,b]$ to be our time interval. We take $\Sigma$ to be a subset of the boundary $\partial\Omega$. 
Let $X=H^1(\Omega)$ be the state space and $\mathcal{O}$ a Hilbert space called the observation space. 
In this section, we present some preliminary material regarding the Caputo fractional derivative, 
the trace operator as well as its right inverse, the analytic semigroup, and the restriction operator. 
Moreover, we define a weighted Lebesgue space to be used later. Let us start by giving 
a brief reminder of some function spaces.\\
- The space of continuous functions from $(0,b)$ to $X$: 
$$ 
C(0,b;X) = \left\{ f : (0,b) \longrightarrow X \ | \ f \mbox{ is continuous} \right\}. 
$$
- The space of absolutely continuous functions from $(0,b)$ to $X$: 
$$ 
AC(0,b;X) = \left\{ f : (0,b) \longrightarrow X \ | \ f \mbox{ is absolutely continuous} \right\}. 
$$
- The Lebesgue spaces $L^r(0,b;X)$, with $r>0$: 
$$ 
L^r(0,b;X) = \left\{ f : (0,b) \longrightarrow X \ \mid \ \int_{0}^{b} \|f(t)\|_X^r dt <+\infty \right\}. 
$$
This space is a Banach space with the norm $\|f\|_{_{L^r(0,b;X)}} = \left(\int_{0}^{b} \|f(t)\|_X^r dt\right)^\frac1r$.

Let us now define the left-sided Caputo time-fractional derivative. By time-fractional derivative, 
we mean that the differentiation is with respect to the time variable.

\begin{definition}\cite{RegAnal}
Let $z$ be in $AC(0,b;X)$. We call the left-sided time-fractional derivative, 
of $z$, of order $\alpha$, in the sense of Caputo, the following formula:	
\begin{equation*}
^{^C}\mathcal{D}^{^\alpha}_{_{0^+}} z(t,x) = \dfrac{1}{\Gamma(1-\alpha)}
\displaystyle\int_{0}^{t}(t-s)^{-\alpha}\dfrac{\partial}{\partial s}z(s,x)ds, 
\quad (t,x)\in]0,b]\times \Omega, \ \alpha \in ]0,1[. 
\end{equation*}
If $\alpha = 1$, then: 
$$
^{^C}\mathcal{D}^{^1}_{_{0^+}} z(t,x) = \dfrac{\partial}{\partial t}z(x,t).
$$
\end{definition}

We now define the trace operator and give a property regarding its right inverse. 

\begin{definition}\cite{adams}
We call the trace operator, of order zero, on $\partial\Omega$, the following operator
\begin{equation*}
\begin{array}{lllll}
\gamma_0& : & H^1(\Omega)& \longrightarrow & H^{\frac{1}{2}}(\partial\Omega) \\
& & \hfill u& \longmapsto&u_{|_{\partial\Omega}},
\end{array}
\end{equation*}
\end{definition}

\begin{proposition}\label{R}\cite{adams}
The trace operator has a right inverse, denoted by $\mathcal{R}$, that is,
$$ \forall v\in H^{\frac{1}{2}}(\partial\Omega) , \quad \gamma_0(\mathcal{R}v) = v.$$
\end{proposition}

As for the restriction operator, we give the following definition.

\begin{definition}\cite{boun.semi1}
The restriction operator, from $\partial\Omega$, on $\Sigma$, is defined as follows,
\begin{equation*}
\begin{array}{lllll}
\chi_{_{\Sigma}}& : & H^{\frac{1}{2}}(\partial\Omega)& \longrightarrow & H^{\frac{1}{2}}(\Sigma) \\
& & \hfill v& \longmapsto&v_{|_{\Sigma}},
\end{array}
\end{equation*}
\end{definition}

The upcoming definition concerns the analytic semigroup.

\begin{definition}\cite{geo}
We say that a $C_0-$semigroup, $\left\{Q(t)\right\}_{t\geq 0}$, is analytic if:\\
The mapping $ t \xrightarrow{\hspace*{0.5cm}} Q(t)x $ is analytic, for all $x$ in $X$. 
\end{definition}

Let us consider $q\geq 1$ and $\beta \leq 0$, the following weighted Lebesgue space,
$$ 
L^{^q}_{_\beta}[0,b] := \left\{ f : [0,b] \longrightarrow \mathbb{K}  \mbox{ measurable }  \ \Big| \ 
\displaystyle\int_{0}^{b}|t^\beta f(t)|^qdt \leq +\infty \right\},
$$
endowed with norm,
$$
\|f\|_{_{L^{^q}_{_\beta}[0,b]}} := \left[\displaystyle\int_{0}^{b}|t^\beta f(t)|^qdt\right]^{\frac{1}{q}},
$$
is a Banach space, see \cite{adams}.


\section{Considered System and Regional Boundary Observability}

In this section, we give the general formulation of the considered problem 
as well as all the necessary components for a better understanding. Moreover, 
we shed light on the studied concept, namely regional boundary observability, 
and we give the main results of this paper.\\
Let us consider the following time-fractional system.
\begin{equation}\label{sys.semi.lin}
\left\{\begin{array}{lll}
^{^C}\mathcal{D}_{_{0^+}}^{^\alpha} z(x,t) = \mathcal{A}z(x,t)  + Sz(x,t)  & in
\ \Omega\times]0,b] , \\
\dfrac{\partial z}{\partial\nu_\mathcal{A}}(\xi,t) = 0 & on \ \partial\Omega\times]0,b], \\ 
z(x,0) = z_0(x) & in \ \Omega,
\end{array}  \right.
\end{equation}
and we call its linear part, the following system,
\begin{equation}\label{sys.lin}
\left\{\begin{array}{lll}
^{^C}\mathcal{D}_{_{0^+}}^{^\alpha} z(x,t) = \mathcal{A}z(x,t)   & in
\ \Omega\times]0,b] , \\
\dfrac{\partial z}{\partial\nu_\mathcal{A}}(\xi,t) = 0 & on \ \partial\Omega\times]0,b], \\ 
z(x,0) = z_0(x) & in \ \Omega,
\end{array}  \right.
\end{equation}
where $(\mathcal{A},\mathcal{D}(A))$ is a linear, second-order, differential 
operator which generates an analytic semigroup, 
$\left\{\mathcal{Q}(t)\right\}_{t\geq 0}$ on $X$, 
$\dfrac{\partial z}{\partial\nu_\mathcal{A}}(\xi,t)$ is the co-normal derivative 
of $z$ with respect to the operator $\mathcal{A}$ \cite{lions.mag}, 
$z_0$ is the unknown initial state, and 
$S : L^2(0,b;X) \xrightarrow{\hspace*{0.5cm}} L^2(0,b;X)$, is a nonlinear operator.\\
To carry on with our work, we will denote from now on, without any loss of generality 
and if there is no confusion,  $z(t) := z(.,t)$ whenever the space variable is not relevant.\\
The following definition gives the expression of the mild solution for the system \eqref{sys.semi.lin}.

\begin{definition}\cite{Mu.2017}
We call a mild solution of system \eqref{sys.semi.lin}, any function $z\in C(0,b;X)$ 
that satisfies the following equation,
\begin{equation}\label{sol.semi.lin}
z(t) = \mathcal{P}_\alpha(t)z_0 
+ \displaystyle\int_{0}^{t}(t-s)^{\alpha-1}K_\alpha(t-s)Sz(s)ds, \ t\in ]0,b],
\end{equation}
where \quad $\mathcal{P}_\alpha(t)y 
= \displaystyle\int_{0}^{\infty}\varpi_\alpha(\theta)\mathcal{Q}(t^\alpha\theta)yd\theta 
$\label{1}, \quad and \quad $K_\alpha(t)y 
= \alpha\displaystyle\int_{0}^{\infty}\theta\varpi_\alpha(\theta)
\mathcal{Q}(t^\alpha\theta)yd\theta $, $\forall y\in X$,\\
with 
\begin{equation}\label{wrt}
\varpi_\alpha (\theta) = \displaystyle\sum_{n=1}^{\infty}
\dfrac{(-\theta)^{n-1}}{\Gamma(n)\Gamma(1-\alpha n)} , \quad \theta \geq 0, 
\end{equation}
is the Mainardi function.  
\end{definition}

The operators $\mathcal{P}_\alpha$ and $K_\alpha$, which map $X$ into itself, are linear, 
bounded, and strongly continuous, see \cite{Mu.2017}.\\
For simplicity reasons, we denote for the rest of this 
paper $\left(K_\alpha\ast Sz\right)(t) = \displaystyle\int_{0}^{t}(t-s)^{\alpha-1}
K_\alpha(t-s)Sz(s)ds.$ Hence, the mild solution of system \eqref{sys.semi.lin} can be written:
\begin{equation}
\label{sol.semi.lin.new}
z(t) = \mathcal{P}_\alpha(t)z_0 + \left(K_\alpha\ast Sz\right)(t), \ t\in ]0,b].
\end{equation}
The output equation is given as follows,
\begin{equation}\label{sys.out}
\varrho(t) = Cz(t),
\end{equation}
where, $C : X \xrightarrow{\hspace*{0.5cm}}\mathcal{O}$ is an admissible 
observation operator for $\mathcal{P}_\alpha$, that is: 
$$
\exists M>0 \ \mbox{ such that } \ \displaystyle\int_{0}^{b}\|
C\mathcal{P}_\alpha(t)y\|^2_{_\mathcal{O}}dt 
\leq M\|y\|_{_X}^2, \quad \forall y\in \mathcal{D}(A).
$$

Let $\Sigma$ be included in the boundary $\partial\Omega$, called the desired boundary 
sub-region. The primary objective of this paper is to recover the initial state value 
on $\Sigma$. We have the following definition regarding the regional 
boundary observability of the system \eqref{sys.semi.lin}.

\begin{definition}
We say that the system \eqref{sys.semi.lin}, augmented with \eqref{sys.out}, 
is $\mathcal{B}$-observable on $\Sigma$ ($\mathcal{B}$ stands for boundary), 
if we can reconstruct the initial state on the sub-region $\Sigma$.
\end{definition}

It appears that, when working on the boundary, things tend to be sensitive and not necessarily rigorous. 
Thus, direct reconstruction of the initial state on the boundary often leads to instability or noise 
amplification or does not give any result at all. We shall present an alternative way to tackle this problem. 
In order to continue our study, we shall designate a way that allows us to establish a relationship 
or connection between regional internal observability and regional boundary observability. 
This is possible by forming an internal sub-region, denoted $\omega\subset\Omega$, such that 
the desired boundary sub-region is included in its boundary, i.e. $\Sigma\subset\partial\omega$. 
Afterwards, instead of reconstructing the initial state on $\Sigma$, we reconstruct it in $\omega$ 
and we compute its projection on the boundary in order to get its value on $\Sigma$. 
   
Let $p>0$ be a small enough number for the current context, and we define:
$$ 
V_p = \underset{o\in \Sigma}{\bigcup}B(o,p) \ \mbox{ and } \ \omega = V_p\bigcap\Omega ,
$$ 
where $B(o,p)$ is the ball centered at $o$ and whose radius is $p$.

\begin{remark}
It is apparent that $\Sigma\subset\partial\omega$.
\end{remark}
 
We call the observability operator, which plays an important role in defining regional 
observability for linear systems, the following operator, 
$\Pi_\alpha : X \xrightarrow{\hspace*{0.5cm}} L^2(0,b;\mathcal{O})$, 
defined as follows, $(\Pi_\alpha u)(t) = C\mathcal{P}_\alpha(t)u, 
\ \forall u\in X,$ and we say that the linear system \eqref{sys.lin}, 
augmented with \eqref{sys.out}, is approximately $\omega$-observable if, and only if, 
\begin{equation}
\label{omega_p}
\mathcal{K}er\left(\Pi_\alpha\chi_{_{\omega}}^*\right) = \left\{0\right\}
\end{equation} 
with
\begin{equation*}
\begin{array}{lllll}
\chi_{_{\omega}} & : & H^1(\Omega) & \longrightarrow & H^1(\omega) \\
& & \hfill u & \longmapsto & u_{|_{\omega}},
\end{array}
\end{equation*}
the restriction operator in $\omega$, and where
\begin{equation*}
\begin{array}{lllll}
\chi_{_{\omega}}^*& : & H^1(\omega)& \longrightarrow&H^1(\Omega) \\
& & \hfill U & \longmapsto & \left\{\begin{array}{lll}
U & in & \omega,\\
0 & in & \Omega\setminus\omega,
\end{array}\right. 
\end{array}
\end{equation*}
is its adjoint.\\

\begin{remark}\label{chi}
For every $u$ in $X$, we have 
$\chi_{_{\omega}}^*\chi_{_{\omega}}u 
= \left\{\begin{array}{lll}
u & in & \omega\\
0 & in & \Omega\setminus\omega.
\end{array}\right.$
\end{remark}

Let us denote $\Pi_\alpha^{\omega} := \Pi_\alpha\chi_{_{\omega}}^*$. 
If system \eqref{sys.lin} is approximately $\omega$-observable, then, 
by the same arguments in \cite{capetact}, the pseudo (or generalized) 
inverse of $\Pi_\alpha^{\omega}$ is well defined and it is given as follows,
$$
\left[\Pi_\alpha^{\omega}\right]^{^\dagger} 
= \left[\left(\Pi_\alpha^{\omega}\right)^*\left(\Pi_\alpha^{\omega}
\right)\right]^{-1}\left(\Pi_\alpha^{\omega}\right)^*,
$$ 
and satisfies, 
$$ 
\left[\Pi_\alpha^{\omega}\right]^{^\dagger}\Pi_\alpha^{\omega} = Id_{_{H^1(\omega)}}.
$$

Our study will be continued in a novel space denoted $V$ and defined as follows: 
$V := \mathcal{I}m\left(\chi_{_{\omega}}\Pi_\alpha^*\right) 
= \mathcal{I}m\left(\left(\Pi_\alpha^{\omega}\right)^*\right)$. 
This new space, endowed with the norm $\|U\|_V 
= \|\Pi_\alpha\chi_{_{\omega}}^*U\|_{_{L^{^2}(0,b;\mathcal{O})}} 
= \|\Pi_\alpha^{\omega} U\|_{_{L^{^2}(0,b;\mathcal{O})}}$, is a Banach space. 
Moreover, $V$ is continuously embedded in $H^1(\omega)$. In some sense, 
we can say that $V$ is the set of all regionally observable initial states, 
and from now on, the goal is to recover the value of the initial states in $V$.

We will now give a theorem that links the approximate regional observability 
of system \eqref{sys.semi.lin} in $\omega$ and the regional boundary 
observability of system \eqref{sys.lin} on $\Sigma$.

\begin{theorem}
Assuming that system \eqref{sys.lin} is approximately $\omega$-observable, 
then system \eqref{sys.semi.lin} is $\mathcal{B}-$observable on $\Sigma$.  
\end{theorem}

\begin{proof}
Let us first introduce another way to define the trace and restriction operators.

\begin{minipage}{0,5\textwidth}
\begin{equation*}
\begin{array}{lllll}
\tilde{\chi}_{_{\Sigma}} & : & H^{\frac{1}{2}}(\partial\omega) 
& \longrightarrow & H^{\frac{1}{2}}(\Sigma) \\
& & \hfill v & \longmapsto & v_{|_{\Sigma}},
\end{array}
\end{equation*}
\end{minipage}
\begin{minipage}{0,5\textwidth}
\begin{equation*}
\begin{array}{lllll}
\tilde{\gamma}_0 & : & H^1(\omega) & \longrightarrow & H^{\frac{1}{2}}(\partial\omega) \\
& & \hfill U & \longmapsto & U_{|_{\partial\omega}},
\end{array}
\end{equation*}
\end{minipage}
We are interested in recovering the value of $z_0$ on $\Sigma$, which we denote 
$z_0^\Sigma  := \tilde{\chi}_{_{\Sigma}}\tilde{\gamma}_0\chi_{_{\omega}}z_0 = \chi_{_{\Sigma}}\gamma_0z_0$.\\
Let us consider $z_0^{\partial\Omega} := \gamma_0z_0$. Hence $z_0$ can be decomposed in two different ways,	
\begin{equation}\label{decomp}
\begin{array}{llll}
z_0 & = & \mathcal{R}z_0^{\partial\Omega} + z_0^1 \\ 
& = & \chi_{_{\omega}}^*\chi_{_{\omega}}z_0 + z_0^2,
\end{array}
\end{equation}
where $z_0^2 = \left\{\begin{array}{lll}
0 & in & \omega\\
z_0 & in & \Omega\setminus\omega,
\end{array}\right.$
and $z_0^1$ is an element that depends on the operator $\mathcal{R}$ 
such that $z_0^1 = 0$ on $\partial\Omega$.\\
Therefore, 
\begin{equation}
\label{eq1}
\chi_{_{\omega}}^*\chi_{_{\omega}}z_0 = \mathcal{R}z_0^{\partial\Omega} + (z_0^1 - z_0^2).
\end{equation}
The mild solution of system \eqref{sys.semi.lin}, can be written 
$$
z(t) = \mathcal{P}_\alpha(t)\chi_{_{\omega}}^*\chi_{_{\omega}}z_0 
+ \mathcal{P}_\alpha(t) z_0^2 + \left(K_\alpha\ast Sz\right)(t),
$$
by applying the operator $C$ and after rearranging the terms, we get 
$$
\left(\Pi_\alpha^{\omega}\chi_{_{\omega}}z_0\right)(t) 
= \varrho(t) - C\left(\Pi_\alpha z_0^2\right)(t) - C\left(K_\alpha\ast Sz\right)(t),
$$
since system \eqref{sys.lin} is approximately $\omega$-observable, then, 
we can apply the pseudo inverse of $\Pi_\alpha^{\omega}$, thus
$$ 
\chi_{_{\omega}}z_0 = \left[\Pi_\alpha^{\omega}\right]^{^\dagger}
\left(\varrho(\cdot) - C\left(\Pi_\alpha z_0^2\right)(\cdot) 
- C\left(K_\alpha\ast Sz\right)(\cdot)\right), 
$$
hence,
$$ 
\chi_{_{\omega}}^*\chi_{_{\omega}}z_0 
= \chi_{_{\omega}}^*\left[\Pi_\alpha^{\omega}\right]^{^\dagger}\left(\varrho(\cdot) 
- C\left(\Pi_\alpha z_0^2\right)(\cdot) - C\left(K_\alpha\ast Sz\right)(\cdot)\right), 
$$
which gives, by using \eqref{eq1}, 
$$
\mathcal{R}z_0^{\partial\Omega} =  \chi_{_{\omega}}^*\left[\Pi_\alpha^{\omega}
\right]^{^\dagger}\left(\varrho(\cdot) - C\left(\Pi_\alpha z_0^2\right)(\cdot) 
- C\left(K_\alpha\ast Sz\right)(\cdot)\right) - (z_0^1 - z_0^2).
$$
We deduce that 
$$
z_0^\Sigma = \chi_{_{\Sigma}}\gamma_0\mathcal{R}z_0^{\partial\Omega} 
=  \chi_{_{\Sigma}}\gamma_0\left(\chi_{_{\omega}}^*\left[\Pi_\alpha^{\omega}
\right]^{^\dagger}\left(\varrho(\cdot) - C\left(\Pi_\alpha z_0^2\right)(\cdot) 
- C\left(K_\alpha\ast Sz\right)(\cdot)\right) - (z_0^1 - z_0^2)\right).
$$
Finally, system \eqref{sys.semi.lin} is $\mathcal{B}$-observable on $\Sigma$.\\
This completes the proof.
\end{proof}

Let us add the following hypotheses before we give the remaining results of this work:

$\exists q,s,r >0$, satisfying $\dfrac{1}{s} + \dfrac{1}{q} = 1 + \dfrac{1}{r}$ such that
\begin{itemize}
\item[$H_1$ -] $\exists g\in L^{^r}_{_{\alpha-1}}[0,b]$, 
verifying $\|\mathcal{Q}(t^\alpha\theta)\|_{_{\mathcal{L}(X,X^\alpha)}} 
\leq |g(t)|$, $\forall \theta>0, \ \forall t\in [0,b] .$ 
\item[$H_2$ -] The operator $S : L^r(0,b;X^\alpha) \xrightarrow{\hspace*{0.5cm}} 
L^s(0,b;X)$ is well defined and satisfies
\begin{equation}
\label{cnd.N}
\left\{
\begin{array}{lll}
||Sz_1-Sz_2||_{_{L^{^s}\left(0,b;X\right)}} \leq k(||z_1||,||z_2||)||z_1-z_2||_{_{L^{^r}(0,b;X^\alpha)}}  ,\\
\text{with } k:]0,+\infty[\times]0,+\infty[ \longrightarrow ]0,+\infty[ \text{ such that } 
k(t_1,t_2) \xrightarrow[t_1,t_2\rightarrow 0]{} 0\\
S(0) = 0,
\end{array}
\right.
\end{equation}
where $X^\alpha$ is the fractional power of $X$, see \cite{pazy} for more information.	
\end{itemize} 

The assumption $H_1$ results in many useful properties regarding the norms 
of $\mathcal{P}_\alpha$ and $\mathcal{K}_\alpha$ in $\mathcal{L}(X,X^\alpha)$ 
which will be used to prove the upcoming results of this paper. 
So, based upon $H_1$, we can easily show that:

\begin{equation}
\label{cnd.p.k}
\|\mathcal{P}_\alpha(\cdot)\|_{_{L^{^r}(0,b;X^\alpha)}} 
\leq ||g(\cdot)||_{_{L^{^r}[0,b]}}   
\mbox{ and } ||K_\alpha(\cdot)||_{_{L^{^q}\left(0,b;\mathcal{L}(X,X^\alpha)\right)}} 
\leq ||g(\cdot)||_{_{L^{^q}_{_{\alpha-1}}[0,b]}}.
\end{equation}

Then, we have the following theorems,

\begin{theorem}
Let us define, for all $\tilde{z}_0$ in $V$, the mapping
$$
\Theta_{_{\tilde{z}_0}}(z) = \mathcal{P}_\alpha(\cdot)\chi_{_{\omega}}^*\tilde{z}_0 
+ \left(K_\alpha\ast Sz\right)(\cdot), \quad  \forall z \in L^2(0,b;X).
$$
We have the following assertions:
\begin{itemize}
\item[1-] $\exists a,m>0$ such that, $\forall \tilde{z}_0 \in B(0,m)\subset V$, 
the functional $\Theta_{_{\tilde{z}_0}}$ has a unique fixed point which 
is the unique mild solution of \eqref{sol.semi.lin} in $B(0,a) \subset L^r(0,b;X^\alpha)$.
\item[2-] The functional
$$
\begin{array}{lllll}
h& : & B(0,m) & \longrightarrow & B(0,a)\\
&   &   \hfill \tilde{z}_0 & \longmapsto & z(\cdot),
\end{array}
$$  
that associates to every regional initial state in $\omega$, $\tilde{z}_0$, 
the corresponding unique mild solution of \eqref{sys.semi.lin}, is globally Lipschitz.
\end{itemize}
\end{theorem}

\begin{proof}
\begin{itemize}
\item[1.] Since $\lim\limits_{t_1,t_2\rightarrow 0} k(t_1,t_2)= 0$, 
then, there exists $a>0$ such that:
\begin{equation}
\label{C1}
\underset{t_i\leq a}{Sup}\ k(t_1,t_2) 
< \dfrac{\Gamma(\alpha)}{\|g\|_{_{L^q_{\alpha-1}[0,b]}}}.
\end{equation}
The idea to prove the first result of this theorem is to use Banach's fixed point theorem. 
This will be done in two steps.\\
\textit{Step 1.} Let us show that $\Theta_{_{\tilde{z}_0}}$ 
is a strict contraction.\\
Let $z_1$ and $z_2$ be two elements of $B(0,a)$. We have:
$$
\begin{array}{lllll}
||\Theta_{_{\tilde{z}_0}}(z_1) - \Theta_{_{\tilde{z}_0}}(z_2)||_{_{L^{^r}(0,b;X^\alpha)}}
&= &||K_\alpha\ast(Sz_1-Sz_2)||_{_{L^{^r}(0,b;X^\alpha)}}.
\end{array}
$$
Using Young's inequality in \cite{ardent}, we get:
$$
\begin{array}{lllll}
||\Theta_{_{\tilde{z}_0}}(z_1) - \Theta_{_{\tilde{z}_0}}(z_2)||_{_{L^{^r}(0,b;X^\alpha)}}
&\leq &||K_\alpha(\cdot)||_{_{L^{^q}\left(0,b;\mathcal{L}(X,X^\alpha)\right)}}
\cdot ||Sz_1-Sz_2||_{_{L^{^s}(0,b;X)}},
\end{array}
$$
from \eqref{cnd.p.k} and \eqref{cnd.N}, we get:
$$
\begin{array}{lllll}
||\Theta_{_{\tilde{z}_0}}(z_1) - \Theta_{_{\tilde{z}_0}}(z_2)||_{_{L^{^r}(0,b;X^\alpha)}}
&\leq &\dfrac{||g(\cdot)||_{_{L^{^q}_{_{\alpha-1}}[0,b]}}}{\Gamma(\alpha)} 
k(||z_1||,||z_2||)||z_1-z_2||_{_{L^{^r}(0,b;X^\alpha)}}.
\end{array}
$$
Let us denote $\mathcal{C}_1 
:= \dfrac{||g(\cdot)||_{_{L^{^q}_{_{\alpha-1}}[0,b]}}}{\Gamma(\alpha)} k(||z_1||,||z_2||)$, thus:
$$
\begin{array}{lllll}
||\Theta_{_{\tilde{z}_0}}(z_1) - \Theta_{_{\tilde{z}_0}}(z_2)||_{_{L^{^r}(0,b;X^\alpha)}}
&\leq &\mathcal{C}_1||z_1-z_2||_{_{L^{^r}(0,b;X^\alpha)}}.
\end{array}
$$
and from \eqref{C1} we can deduce that $\mathcal{C}_1<1.$ This means that 
$\Theta_{_{\tilde{z}_0}}$ is a strict contraction.\\
\textit{Step 2.} Let us demonstrate that 
$\Theta_{_{\tilde{z}_0}}\left(B(0,a)\right)\subset B(0,a).$\\
Let $z$ be in $B(0,a)$, we have:
$$
\begin{array}{llll}
\|\Theta_{_{\tilde{z}_0}}(z)\|_{_{L^{^r}(0,b;X^\alpha)}} 
& = & \|\mathcal{P}_\alpha(\cdot)\chi_{_{\omega}}^*\tilde{z}_0 
+ (K_\alpha\ast Sz)\|_{_{L^{^r}(0,b;X^\alpha)}}, \\
& \leq & \|\mathcal{P}_\alpha(\cdot)\chi_{_{\omega}}^*\tilde{z}_0\|_{_{L^{^r}(0,b;X^\alpha)}} 
+ \|(K_\alpha\ast Sz)\|_{_{L^{^r}(0,b;X^\alpha)}},\\
& \leq & \|\tilde{z}_0\|_{_V}\|g\|_{_{L^{^r}[0,b]}} 
+ \|K_\alpha\|_{_{L^{^q}\left(0,b;\mathcal{L}(X,X^\alpha)\right)}}\|Sz\|_{_{L^{^s}\left(0,b;X\right)}},\\
& \leq & \|\tilde{z}_0\|_{_V}\|g\|_{_{L^{^r}[0,b]}} 
+ \dfrac{a.\|g\|_{_{L^q_{\alpha-1}[0,b]}}}{\Gamma(\alpha)}\underset{t_1\leq a}{Sup}\ k(t_1,0).
\end{array}
$$
Thus, if $$ \quad m = \dfrac{a}{||g||_{_{L^{^r}[0,b]}}}\left(1 
- \dfrac{||g||_{_{L^{^q}_{_{\alpha-1}}[0,b]}}}{\Gamma(\alpha)} 
\underset{t_1\leq a}{Sup}\ k(t_1,0) \right),
$$ 
it is clear that, 
$$
\tilde{z}_0 \in B(0,m) \implies \Theta_{_{\tilde{z}_0}}(z) \in B(0,a). 
$$
Hence, from Banach's fixed point theorem $\Theta_{_{\tilde{z}_0}}$ has a unique fixed point.
\item[2.] Let us consider $\tilde{z}_{1_{_0}}$ and $\tilde{z}_{2_{_0}}$ 
in $B(0,m)$ such that $h(\tilde{z}_{1_{_0}}) = z_1$ and $h(\tilde{z}_{2_{_0}}) = z_2$.\\
$$
\begin{array}{llll}
\|h(\tilde{z}_{1_{_0}}) - h(\tilde{z}_{2_{_0}})\|_{_{L^{^r}(0,b;X^\alpha)}} 
& = & \|\mathcal{P}_\alpha(\cdot)\chi_{_{\omega}}^*\left(\tilde{z}_{1_{_0}}-\tilde{z}_{2_{_0}}\right) 
+ K_\alpha\left(Sz_1-Sz_2\right)\|_{_{L^{^r}(0,b;X^\alpha)}},\\
& \leq & \|g\|_{_{L^{^r}[0,b]}}\| \tilde{z}_{1_{_0}} - \tilde{z}_{2_{_0}} \|_{_V} 
+ \mathcal{C}_1\|z_1-z_2\|_{_{L^{^r}(0,b;X^\alpha)}}.
\end{array}
$$
This yields that:
\begin{equation}\label{th.c1}
\|h(\tilde{z}_{1_{_0}}) - h(\tilde{z}_{2_{_0}})\|_{_{L^{^r}(0,b;X^\alpha)}} 
\leq \dfrac{\|g\|_{_{L^{^r}[0,b]}}}{1 - \mathcal{C}_1}\| \tilde{z}_{1_{_0}} 
- \tilde{z}_{2_{_0}} \|_{_V}.
\end{equation}
Finally, $h$ is globally Lipschitz.
\end{itemize}
This completes the proof.
\end{proof}

\begin{theorem}
Let us consider, for every $\varrho$ in $L^2(0,b;\mathcal{O})$, 
the mapping $$\vartheta_\varrho(\tilde{z}_0) 
= \left[\Pi_\alpha^{\omega}\right]^{^\dagger}\left[\varrho(\cdot) 
- C\left(K_\alpha\ast Sh(\tilde{z}_0)\right)(\cdot)\right], 
\quad \forall \tilde{z}_0\in B(0,m).$$
If, moreover, the following assertions,
\begin{itemize}
\item[$H_3$ -] $\forall\tilde{z}_0\in B(0,m)$, 
$C\left(K_\alpha\ast Sh(\tilde{z}_0)\right)(\cdot)\in\mathcal{I}m(\Pi_\alpha^{\omega}).$
\item[$H_4$ -] $\exists \delta>0,$ $\|C\left(K_\alpha\ast Sz)\right)
(\cdot)\|_{_{L^2(0,b;\mathcal{O})}} \leq \delta\|Sz\|_{_{L^s(0,b;X)}}, \quad \forall z\in B(0,a).$
\end{itemize}
are satisfied, then we have :
\begin{itemize}
\item[1-] $\exists \rho>0, \ \forall \varrho\in B(0,\rho)\subset L^2(0,b;\mathcal{O})$, 
$\vartheta_\varrho(\cdot)$ admits a unique fixed point in $B(0,m)$.
\item[2- ] The function,
$$\begin{array}{lllll}
l & : & B(0,\rho) & \longrightarrow & B(0,m)\\
&   &   \hfill  \varrho& \longmapsto & \tilde{z}_0,
\end{array}$$  
that corresponds for every measurement $\varrho$ in $B(0,\rho)$ the associated unique 
fixed point of $\vartheta_\varrho(\cdot)$, which is the initial state in $\omega$, is globally Lipschitz.
\end{itemize}
\end{theorem}

\begin{proof}
\begin{itemize}
\item[1.] From $\lim\limits_{t_1,t_2\rightarrow 0} k(t_1,t_2)= 0$, 
we can deduce that, $\exists \rho$ such that:
\begin{equation}
\label{C2}
\underset{t_i\leq l}{Sup}\ k(t_1,t_2) 
< \dfrac{1-\mathcal{C}_1}{\delta\|g\|_{_{L^r[0,b]}}}.
\end{equation}
Similarly to the first assertion in the previous theorem, we shall 
use Banach's fixed point theorem.\\
\textit{Step 1.} Let $\tilde{z}_0$ and $\bar{z}_0$ be in $B(0,m)$, then:
$$
\begin{array}{lll}
\| \vartheta_\varrho(\tilde{z}_0) - \vartheta_\varrho(\bar{z}_0) \|_V 
& = & \|\left[\Pi_\alpha^{\omega}\right]^{^\dagger}\left[C\left(K_\alpha\ast 
\left(Sh(\tilde{z}_0) - Sh(\bar{z}_0)\right)\right)\right]\|_V\\\\
& = & \|\Pi_\alpha^{\omega}\left[\Pi_\alpha^{\omega}\right]^{^\dagger}\left[
C\left(K_\alpha\ast \left(Sh(\tilde{z}_0) 
- Sh(\bar{z}_0)\right)\right)\right]\|_{_{L^{^2}(0,b;\mathcal{O})}}.
\end{array}
$$
Using $H_3$, we can obtain that,
$$
\begin{array}{lll}
\| \vartheta_\varrho(\tilde{z}_0) - \vartheta_\varrho(\bar{z}_0) \|_V 
& \leq  &  \|C\left(K_\alpha\ast \left(Sh(\tilde{z}_0) 
- Sh(\bar{z}_0)\right)\right)\|_{_{L^{^2}(0,b;\mathcal{O})}},
\end{array}
$$
and from $H_4$, we get,
$$
\begin{array}{lll}
\| \vartheta_\varrho(\tilde{z}_0) - \vartheta_\varrho(\bar{z}_0) \|_V 
& \leq  &  \delta\|Sh(\tilde{z}_0) - Sh(\bar{z}_0)\|_{_{L^{^2}(0,b;\mathcal{O})}},
\end{array}
$$
with the help of $H_2$ and $\eqref{th.c1}$, we can see that,
$$
\begin{array}{lll}
\| \vartheta_\varrho(\tilde{z}_0) - \vartheta_\varrho(\bar{z}_0) \|_V 
& \leq  &\underset{t_i\leq l}{Sup}\ k(t_1,t_2)\dfrac{\delta\|g\|_{_{L^r[0,b]}}}{1
-\mathcal{C}_1}\|\tilde{z}_0 - \bar{z}_0\|_V.
\end{array}
$$
So, if we set $\mathcal{C}_2 := \underset{t_i\leq l}{Sup}\ k(t_1,t_2)
\dfrac{\delta\|g\|_{_{L^r[0,b]}}}{1-\mathcal{C}_1}$, we can write,
$$
\begin{array}{lll}
\| \vartheta_\varrho(\tilde{z}_0) - \vartheta_\varrho(\bar{z}_0) \|_V 
& \leq  &\mathcal{C}_2\|\tilde{z}_0 - \bar{z}_0\|_V,
\end{array}
$$
and from \eqref{C2}, we can see that $\mathcal{C}_2<1$. 
Thus, $\vartheta_\varrho$ is a strict contraction on $B(0,m)$.\\
\textit{Step 2.} Let us show that 
$\vartheta_\varrho\left(B(0,m)\right)\subset B(0,m)$.\\
Let $\tilde{z}_0$ be in $B(0,m)$, we have:
$$
\begin{array}{llll}
\|\vartheta_\varrho(\tilde{z}_0)\|_V
& = & \|\varrho(\cdot) - C\left(K_\alpha\ast Sh(\tilde{z}_0)
\right)\|_{_{L^{^2}(0,b;\mathcal{O})}},\\
& \leq & \|\varrho\|_{_{L^{^2}(0,b;\mathcal{O})}} 
+ \delta.k(\|h(\tilde{z}_0)\|,0)\|h(\tilde{z}_0)\|_{_{L^{^r}(0,b;X^\alpha)}},\\
& \leq & \|\varrho\|_{_{L^{^2}(0,b;\mathcal{O})}} 
+ a\delta\underset{t\leq l}{Sup}\ k(t,0).
\end{array}
$$
So, if we choose $\rho := m - a\delta\underset{t\leq l}{Sup}\ k(t,0)$, 
then it is obvious that $\varrho\in B(0,\rho)$ implies that 
$\vartheta_\varrho(\tilde{z}_0)\in B(0,\rho)$. Hence, $\vartheta_\varrho(B(0,m))\subset B(0,m)$.\\
Thus, from Banach's fixed point theorem $\vartheta_\varrho$ has a unique fixed point.
\item[2.] Let $\varrho_1$ and $\varrho_2$ be in $B(0,\rho)$ such that $l(\varrho_1) 
= \tilde{z}_{{1_{_0}}} = \vartheta_{\varrho_1}(\tilde{z}_{{1_{_0}}})$ and $l(\varrho_2) 
= \tilde{z}_{{2_{_0}}} = \vartheta_{\varrho_2}(\tilde{z}_{{2_{_0}}})$.\\
$$
\begin{array}{llll}
\|l(\varrho_1) - l(\varrho_2)\|_V
& = & \|\vartheta_{\varrho_1}(\tilde{z}_{{1_{_0}}})-\vartheta_{\varrho_2}(\tilde{z}_{{2_{_0}}})  \|_V,\\
& \leq & \|\vartheta_{\varrho_1}(\tilde{z}_{{1_{_0}}})-\vartheta_{\varrho_1}(\tilde{z}_{{2_{_0}}})  \|_V 
+ \|\vartheta_{\varrho_1}(\tilde{z}_{{2_{_0}}})-\vartheta_{\varrho_2}(\tilde{z}_{{2_{_0}}})  \|_V,\\
& \leq & \mathcal{C}_2\|l(\varrho_1) - l(\varrho_2)\|_V + \|\varrho_1-\varrho_2\|_{_{L^{^2}(0,b;\mathcal{O})}}.
\end{array}
$$
This means that,
$$
\|l(\varrho_1) - l(\varrho_2)\|_V \leq \dfrac{1}{1 - \mathcal{C}_2}\|
\varrho_1-\varrho_2\|_{_{L^{^2}(0,b;\mathcal{O})}}.
$$
Hence, $l$ is globally Lipschitz.
\end{itemize}
This completes the proof.
\end{proof}

The first theorem indicates that for every regional initial state, in a certain 
ball of $V$, system \eqref{sys.semi.lin} has a unique mild solution, in a certain 
ball in $L^r(0,b;X^\alpha)$, which is a fixed point of $\Theta_{_{\tilde{z}_0}}$. 
As for the second theorem, it shows that for every measurement, taken on the system 
\eqref{sol.semi.lin}, in a ball of $L^2(0,b;\mathcal{O})$, we obtain a unique regional 
initial state in ball of $V$, which is a fixed point of $\vartheta_\varrho(\cdot)$ 
that corresponds to the initial state in $\omega$.      

Using the definition of $\vartheta_\varrho$, we give the following theorem, 
in which we give a sequence that converges to the initial state in the internal 
sub-region $\omega$.

\begin{theorem}
Let $n$ be in $\mathbb{N}$, the following sequence,
$$\left\{\begin{array}{llll}
\tilde{z}_0^0 = 0\\
\tilde{z}_0^{n+1} = \left[\Pi_\alpha^{\omega}\right]^{^\dagger}\left[\varrho(\cdot)
- C\left(K_\alpha\ast Sh(\tilde{z}_0^n)\right)(\cdot)\right], \ \forall n\geq 0.
\end{array}\right.$$
converges to the initial state in $\omega$.
\end{theorem}

\begin{proof}
Let us first show that $\left(\tilde{z}_0^{n}\right)_n$ is convergent. For that, 
we will show that it is a Cauchy sequence in a Banach space, hence convergent.\\
Let $n$ be in $\mathbb{N}$, we have:
$$\begin{array}{lcl}
\|\tilde{z}_0^{n+1} - \tilde{z}_0^{n}\|_V 
& = & \|\vartheta_\varrho(\tilde{z}_0^{n}) - \vartheta_\varrho(\tilde{z}_0^{n-1})\|_V,\\
& \leq & \mathcal{C}_2\|\tilde{z}_0^{n} - \tilde{z}_0^{n-1}\|_V,\\
&\vdots & \qquad \vdots\\
& \leq & . \mathcal{C}_2^n\|\tilde{z}_0^{1}\|_V.
\end{array}$$
Now let $m$ such that $m>n$. We have:
$$
\begin{array}{lll}
\|\tilde{z}_0^{m} - \tilde{z}_0^{n}\|_V& \leq &\|\tilde{z}_0^{m} 
- \tilde{z}_0^{m-1}\|_V +\|\tilde{z}_0^{m-1} - \tilde{z}_0^{m-2}\|_V 
+ \hdots + \|\tilde{z}_0^{n+1} - \tilde{z}_0^{n}\|_V,\\
& \leq & \mathcal{C}_2^{m-1}\|\tilde{z}_0^{1}\|_V + \mathcal{C}_2^{m-2}\|
\tilde{z}_0^{1}\|_V+\hdots + \mathcal{C}_2^{n}\|\tilde{z}_0^{1}\|_V,\\
& \leq & \dfrac{\mathcal{C}_2^{n}}{1-\mathcal{C}_2}\|\tilde{z}_0^{1}\|_V.
\end{array}
$$
Thus, 
$$
\lim\limits_{n,m\rightarrow +\infty}\|\tilde{z}_0^{m} - \tilde{z}_0^{n}\|_V =0.
$$
Hence,  $\left(\tilde{z}_0^{n}\right)_n$ is a Cauchy sequence in $V$, thus convergent.\\
Furthermore, we have:
$$
\|\tilde{z}_0^{n} - \tilde{z}_0\|_V = \|l(\varrho_n) - l(\varrho)\|_V,
$$
such that $\varrho_n = Ch(\tilde{z}_0^n).$\\
Thus, 
$$ 
\begin{array}{lll}
\|\tilde{z}_0^{n} - \tilde{z}_0\|_V 
& \leq & \dfrac{1}{1 - \mathcal{C}_2}\|\varrho_n-\varrho\|_{_{L^{^2}(0,b;\mathcal{O})}},\\
& \leq &  \dfrac{1}{1 - \mathcal{C}_2}\|\varrho - \Pi_\alpha^{\omega}\tilde{z}_0^n 
- C\left(K_\alpha\ast Sh(\tilde{z}_0^n)\right) \|_{_{L^{^2}(0,b;\mathcal{O})}},\\
& \leq &  \dfrac{1}{1 - \mathcal{C}_2}\|\Pi_\alpha^{\omega}\tilde{z}_0^{n+1}- 
\Pi_\alpha^{\omega}\tilde{z}_0^n\|_{_{L^{^2}(0,b;\mathcal{O})}},\\
& = & \dfrac{1}{1 - \mathcal{C}_2}\|\tilde{z}_0^{n+1}- \tilde{z}_0^n\|_V,\\
& \leq & \dfrac{\mathcal{C}_2^n}{1 - \mathcal{C}_2}\|\tilde{z}_0^1\|_V 
\xrightarrow[n\rightarrow +\infty]{} 0. 
\end{array} 
$$
Hence,
$$ 
\tilde{z}_0^n \xrightarrow[n\rightarrow +\infty]{}\tilde{z}_0. 
$$
This completes the proof.
\end{proof}

If we denote by $\tilde{z}_0^*$ the limit of $\left\{\tilde{z}_0^n\right\}_{n\geq 0}$, 
then $z_0^{\Sigma} = \tilde{\chi}_{_\Sigma}\tilde{\gamma}_0\tilde{z}_0^*$, and we have 
the following algorithm that allows us to reconstruct the initial state in $\omega$.


\subsection*{Algorithm}

\begin{algorithm}[H]	
\SetAlgoLined
\vspace*{0.5cm}
\begin{description}
\item[\textbf{Step 1. }] Initialization of Data: 
threshold accuracy $\varepsilon$, the differentiation order $\alpha$, 
the desired boundary sub-region $\Sigma$, the linking internal 
sub-region $\omega$, sensors properties.

\item[\textbf{Step 2. }] Compute  $d(\cdot)=\varrho(\cdot) - C\left(K_\alpha\ast Sh(0)\right).$

\item[\textbf{Step 3. }] Repeat:

\begin{description}
\item[\textbullet] $\tilde{z}_0 = \left[\Pi_\alpha^{\omega}\right]^{^\dagger} d(\cdot).$

\item[\textbullet] $ \mbox{ Solve } ^{C}\mathcal{D}^{^\alpha}_{_{0^+}}z
=\mathcal{A}z+Sz \ ; \ \ y(0) = \chi_{\omega}^*\tilde{z}_0.$

\item[\textbullet] $\tilde{\varrho}(\cdot) = Cz.$

\item[\textbullet] $ d(\cdot) =  \tilde{\varrho}(\cdot) 
- C\left(K_\alpha\ast h(\tilde{z}_0)\right).$
\end{description}
\textbf{Until:} $\quad \|\varrho(\cdot) 
- \tilde{\varrho}(\cdot)\|_{_{L^2(0,b;\mathcal{O})}} \leq \varepsilon.$

\item[\textbf{Step 4.}]  The initial state in $\omega$ is $\tilde{z}_0$ 
and on $\Sigma$ is $\tilde{\chi}_{_\Sigma}\tilde{\gamma}_0\tilde{z}_0.$

\end{description}
\caption{\label{alg1} Solution to the boundary reconstruction problem.}
\end{algorithm}


\section{Numerical Simulation}

This section is devoted to numerical simulations; we shall present two examples 
to show the efficiency of the proposed approach. The first example is given with 
an academic model and with a zonal sensor for the measurements. As for the second 
example, we consider the Fisher-KPP model (KPP stands for Kolmogorov–Petrovsky–Piskunov) 
and we work with a pointwise sensor to collect the measurements. 


\subsubsection*{Example 1}

Let us  choose $\Omega = ]0,1[\times]0,1[$ and $T=2$. We consider the system's 
dynamic to be $ \mathcal{A} = \Delta = \displaystyle\sum_{i=1}^{2}
\dfrac{\partial^2}{\partial x_i^2} $, this operator has a set of 
eigenfunctions $\left\{ \phi_{ij}(x_1,x_2) = \dfrac{2}{\sqrt{\pi(1-\beta_{ij})}}
\cos\left(i x_1\right)\cos\left(j x_2\right), \right\}_{i,j\geq 0}$, which is an 
orthonormal basis of $H^1(\Omega)$, associated with the set of eigenvalues 
$\left\{\beta_{ij} = -\left(i^2 + j^2\right)\pi^2\right\}_{i,j\geq 0}.$\\

Let us consider the following two dimensional time-fractional semilinear system,
\begin{equation}
\label{example}
\left\{\begin{array}{lll}
^{^C}\mathcal{D}_{_{0^+}}^{^\alpha} z(x_1,x_2,t) 
= \Delta z(x_1,x_2,t)  + \displaystyle\sum_{i,j\geq 0 }^{\infty}
\langle z(t),\phi_{ij}\rangle_{_X}^2\phi_{ij}(x_1,x_2)  & in
\ \Omega\times]0,2] , \\
\dfrac{\partial z}{\partial\nu}(\xi_1,\xi_2,t) = 0 & on \ \partial\Omega\times]0,2], \\ 
z(x_1,x_2,0) = z_0(x_1,x_2) & in \ \Omega.
\end{array}  \right.
\end{equation}
The output functional is given by means of a zonal sensor $(G,d)$, 
where $G\subset \Omega$ is the spatial support of the sensor and 
$d\in L^2(G)$ is its spatial distribution. In such case, 
the output equation becomes,
$$	
\varrho(t) = \displaystyle\iint_{G}z(x_1,x_2,t)d(x_1,x_2)dx_1dx_2,
$$
and the observation space is $\mathcal{O} = \mathbb{R}$. For this first example, 
we take $G = \left[0.2\ ,\ 0.3\right]\times\left[0.7\ ,\ 0.9\right]$, $d \equiv 1 $, 
the differentiation order $\alpha = 0.75$, the initial state to be observed

$$ 
z_0(x_1,x_2) = \cos((x_1-1)\pi)\cos((x_2-1)\pi) ,
$$

and the desired boundary sub-region $\Sigma = ]0,1[\times\left\{1\right\}$.\\
After setting all the parameters of the system, we apply the proposed algorithm 
and we get the Figures~\ref{fig:ex1} and \ref{fig:ex1_init_both}. In 
sub-figure \ref{fig:ex1_init}, we find the initial state that we are trying 
to reconstruct; in sub-figure \ref{fig:ex1_init_rec}, we can see the plot of 
the initial state given to us by the proposed algorithm, we also call it the 
recovered or reconstructed initial state. At first sight, it is very clear 
that the two plots in sub-figures \ref{fig:ex1_init} and \ref{fig:ex1_init_rec} 
are different from one another; this is not the case on the boundary sub-region $\Sigma$; 
indeed, if we take a vertical cut at $x_2 = 1$ in both sub-figure we can immediately 
realize that the two initial states are almost identical, this is illustrated in 
Figure~\ref{fig:ex1_init_both} where we can see the plots of the two initial states 
on the desired boundary sub-region. Moreover, the resulting $L^2$-error, which we also 
call the reconstruction error, in this first example is:
$$
\| z_0 - \tilde{z}_0 \|_{_{L^2(\Sigma)}} = 4.2426\mbox{E-} 04.
$$

Taking into account Figure~\ref{fig:ex1_init_both} and the satisfying value 
of the reconstruction error, we can confidently say that the proposed approach 
in this paper is very effective when applied to the regional boundary 
reconstruction problem of the initial state.
The reconstruction error is also controlled by where we put the sensors in our 
evolution domain; indeed, in Table~\ref{tab:1} we can observe that the error 
changes whenever we change the location of the sensor, some positions result 
in an enormous value of the reconstruction error, with the same values in the 
other parameters; these positions, in this case $G = \left[0.7\ ,\ 0.9\right]
\times\left[0.5\ ,\ 0.7\right] $ and $G = \left[0.7\ ,\ 0.9\right]
\times\left[0.3\ ,\ 0.5\right]$, are called non-strategic, meaning that if we 
put the sensors at those locations then the considered initial state is not 
approximately observable, and thus cannot be reconstructed from the measurements 
provided by the sensors. The investigation of the relationship between the sensors' 
location and the reconstruction error for a semilinear system is still currently 
based solely on numerical evidence; to the best of our knowledge there are no 
solid theoretical results regarding this issue. Thus, further work needs to 
be done on the subject. 
\begin{table}[H]
\begin{center}
\begin{tabular}{ | c || c |}
\hline
Spatial support $G$ & Error $\|z_0 - \tilde{z}_0\|_{_{L^2(\Sigma)}}^2$ \\ \hline\hline
$\left[0.2\ ,\ 0.3\right]\times\left[0.7\ ,\ 0.9\right]$& $ 4.2426\mbox{E-} 04$ \\ \hline

$\left[0.3\ ,\ 0.5\right]\times\left[0.7\ ,\ 0.9\right]$&  $5.2666\mbox{E-} 02$ \\ \hline

$\left[0.5\ ,\ 0.7\right]\times\left[0.7\ ,\ 0.9\right]$& $1.3632\mbox{E-} 02$\\ \hline

$\left[0.7\ ,\ 0.9\right]\times\left[0.7\ ,\ 0.9\right]$	&$3.2514\mbox{E-} 03$\\ \hline

$\left[0.7\ ,\ 0.9\right]\times\left[0.5\ ,\ 0.7\right]$	&$1.0214\mbox{E+} 09$\\ \hline	

$\left[0.7\ ,\ 0.9\right]\times\left[0.3\ ,\ 0.5\right]$	&$8.1616\mbox{E+} 06$\\ \hline

$\left[0.7\ ,\ 0.9\right]\times\left[0.1\ ,\ 0.3\right]$	&$1.0023\mbox{E-} 01$\\ \hline

\end{tabular}
\end{center}
\caption{Different values of the $L^2$-error in function 
of the sensor's positioning for the first example.}	
\label{tab:1}
\end{table}

\begin{figure}
\centering
\begin{subfigure}[b]{0.49\textwidth}
\centering
\includegraphics[width=1.1\textwidth]{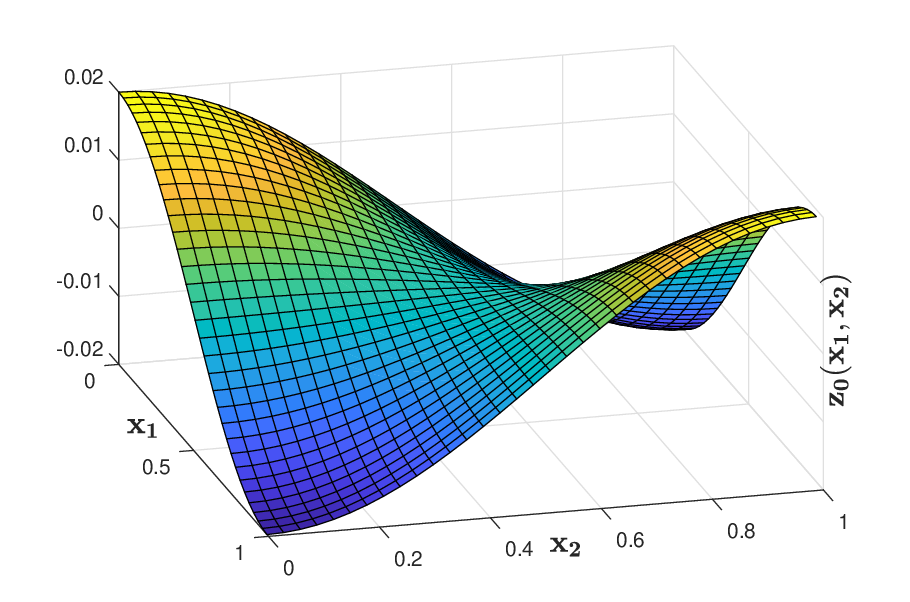}\\
\caption{The initial state $z_0$ in $\Omega$ for the first example.}
\label{fig:ex1_init}
\end{subfigure}
\hfill
\begin{subfigure}[b]{0.49\textwidth}
\centering
\includegraphics[width=1.1\textwidth]{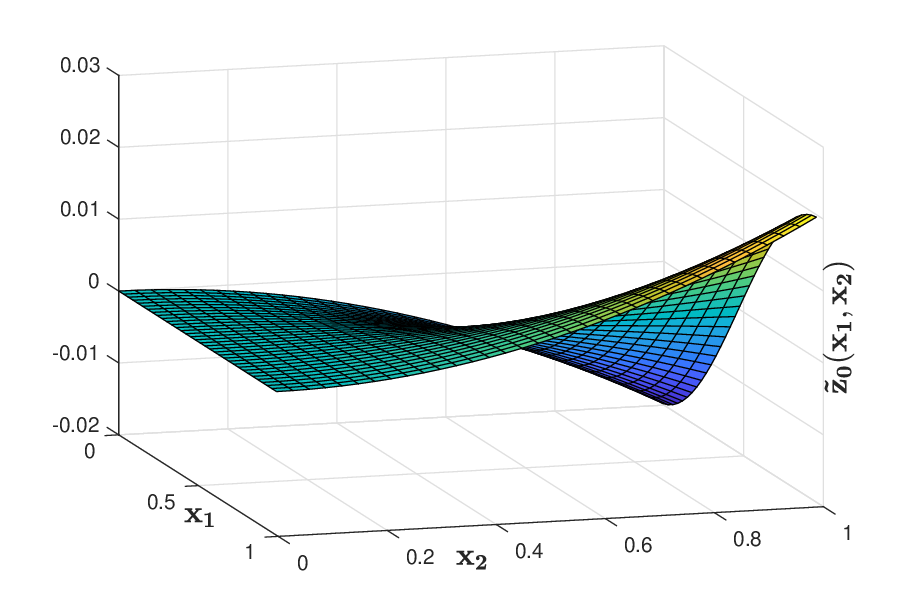}\\
\caption{The recovered initial state $\tilde{z}_0$ in $\Omega$ for the first example.}
\label{fig:ex1_init_rec}
\end{subfigure}
\caption{The real initial state and the recovered one in $\Omega$ for the first example}
\label{fig:ex1}
\end{figure}

\begin{figure}
\center{\includegraphics[width=11cm]{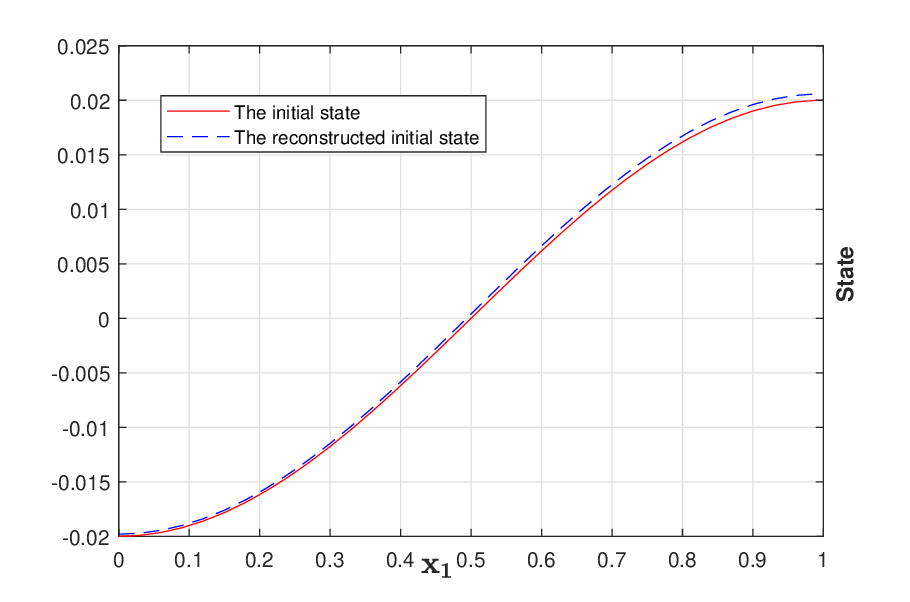}}
\caption{\label{fig:ex1_init_both} {The recovered initial state $\tilde{z}_0$ 
and the real one $z_0$ in the boundary sub-region $\Sigma$ for the first example.}}
\end{figure}


\subsubsection*{Example 2}

For this second example, we consider the Fisher-KPP equation which describes 
the evolution in time of the spread of the density for some epidemics or populations 
diffusing and reacting at the same time. We consider a normalized unit square domain 
$\Omega = ]0,1[\times]0,1[$ and we assume that the system is isolated from outside 
the domain in order to impose homogeneous Newton boundary conditions. 
The system has the following form:
\begin{equation}
\label{example2}
\left\{\begin{array}{lll}
^{^C}\mathcal{D}_{_{0^+}}^{^\alpha} z(x_1,x_2,t) 
= d\Delta z(x_1,x_2,t)  +  rz(x_1,x_2,t)\left(1-z(x_1,x_2,t)\right) 
& in \ \Omega\times]0,1] , \\
\dfrac{\partial z}{\partial\nu}(\xi_1,\xi_2,t) = 0 
& on \ \partial\Omega\times]0,1], \\ 
z(x_1,x_2,0) = z_0(x_1,x_2) & in \ \Omega.
\end{array}  \right.
\end{equation}
 
The parameters of the system are considered $d=0.4$ and $r=0.9$. The measurements 
in this second example are given by a pointwise sensor $(c,\delta_c)$; that is, 
$c = (c_1,c_2)\in \Omega$  is the spatial position of the sensor, $\delta_c$ 
is the Dirac delta mass centered at $c$, the observation space is 
$\mathcal{O}=\mathbb{R}$, and the output function is written,
$$
\varrho(t) = z(c_1,c_2,t).
$$
We set the parameters values as follows, the sensor's position $c = (0.1\ ,\ 0.7)$, 
the differentiation order $\alpha = 0.25$, the desired sub-region 
$\Sigma = \left\{1\right\}\times]0,1[$, and the initial state to be observed,  
$$ 
z_0(x_1,x_2) = (x_1x_2)^2e^{-2x_1}\left(1-x_2\right)^2.
$$

By applying the proposed algorithm, we obtain the Figures~\ref{fig:ex2} and \ref{fig:ex2_init_both}. 
The sub-figures \ref{fig:ex2_init} and \ref{fig:ex2_init_rec} show, respectively, 
the real initial state and the reconstructed initial state. Even though the two plots 
\ref{fig:ex2_init} and \ref{fig:ex2_init_rec} seem to differ hugely from one another, 
they coincide with each other on the desired boundary sub-region. In fact, if we consider 
$x_1 = 1$, we can see in Figure~\ref{fig:ex2_init_both} that the initial state 
(the red continuous plot) and the recovered initial state (the blue dotted plot) 
are nearly identical. The $L^2$-error for the second example is:
$$
\| z_0 - \tilde{z}_0 \|_{_{L^2(\Sigma)}} = 1.7852\mbox{E-} 04.
$$

\begin{figure}
\centering
\begin{subfigure}[b]{0.49\textwidth}
\centering
\includegraphics[width=1.1\textwidth]{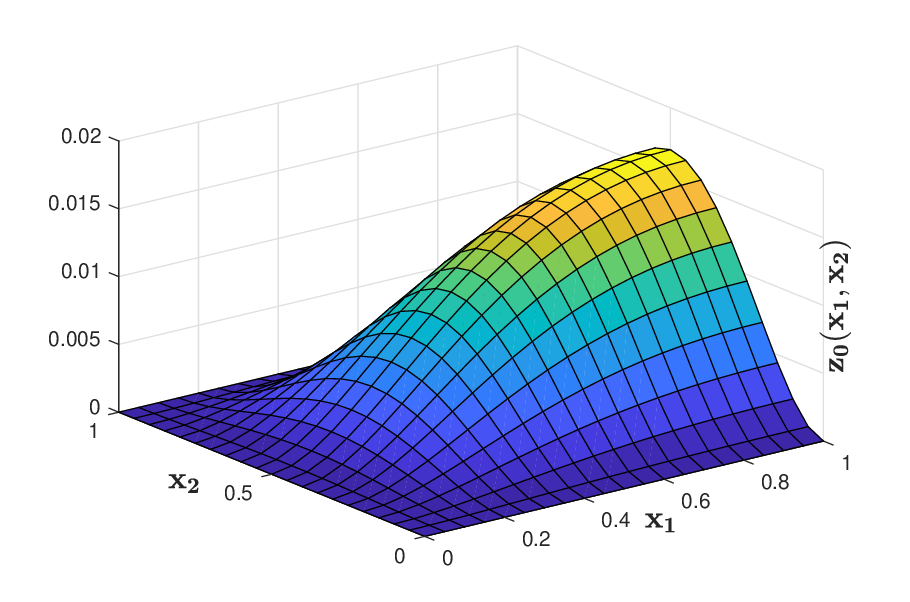}\\
\caption{The initial state $z_0$ in $\Omega$ for the second example.}
\label{fig:ex2_init}
\end{subfigure}
\hfill
\begin{subfigure}[b]{0.49\textwidth}
\centering
\includegraphics[width=1.1\textwidth]{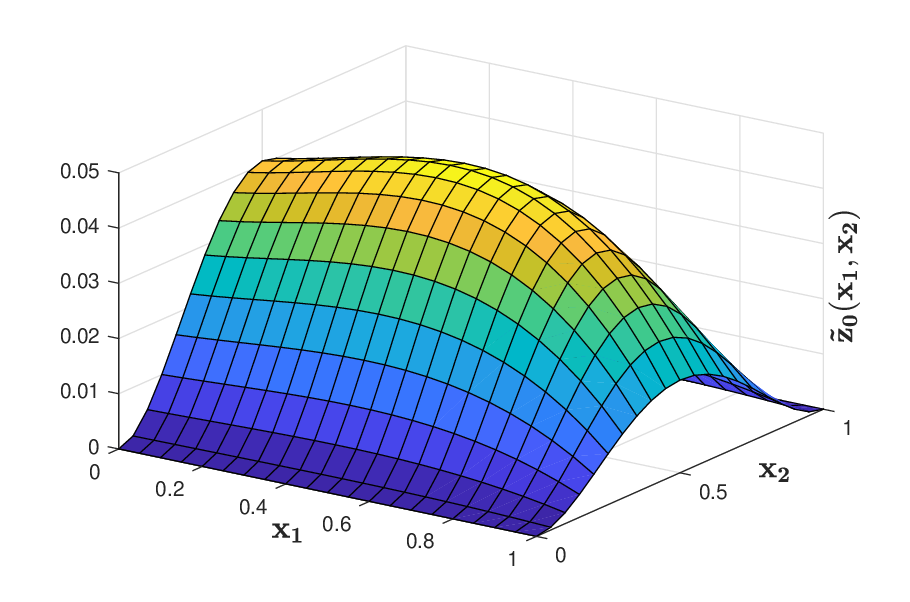}
\\
\caption{The recovered initial state $\tilde{z}_0$ in $\Omega$ for the second example.}
\label{fig:ex2_init_rec}
\end{subfigure}
\caption{The real initial state and the recovered one in $\Omega$ for the second example}
\label{fig:ex2}
\end{figure}

\begin{figure}
\center{\includegraphics[width=11cm]{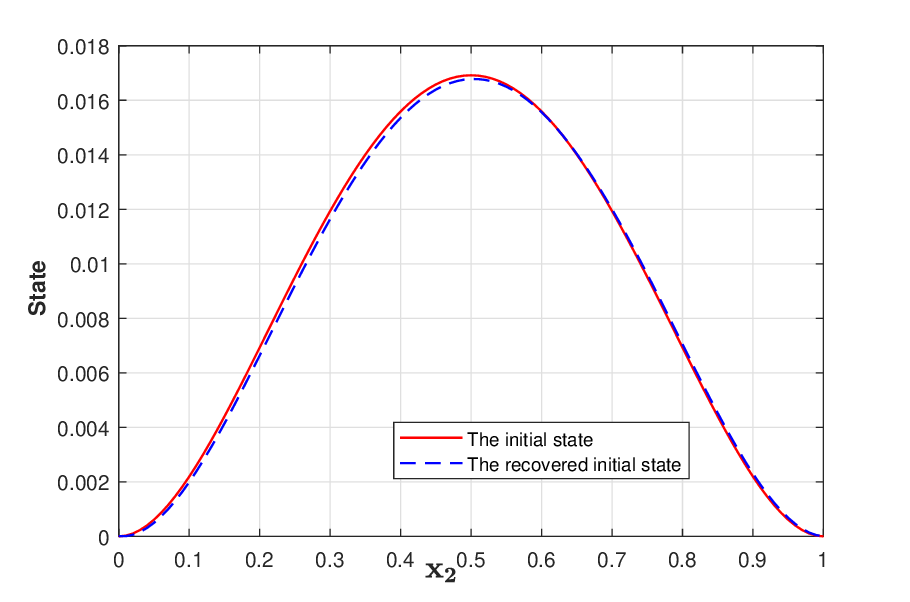}}
\caption{\label{fig:ex2_init_both} {The recovered initial state $\tilde{z}_0$ and the real 
one $z_0$ in the boundary sub-region $\Sigma$ for the Second example.}}
\end{figure}

Similarly to the first example, we can observe that the reconstruction error in 
the second example is also satisfying. Moreover, the same discussion regarding 
the positioning of the sensors and their effect on the reconstruction error 
is applicable, and this is shown in Figure~\ref{fig:ex2_evol}.

\begin{figure}
\centering
\begin{subfigure}[b]{0.49\textwidth}
\centering
\includegraphics[width=1.1\textwidth]{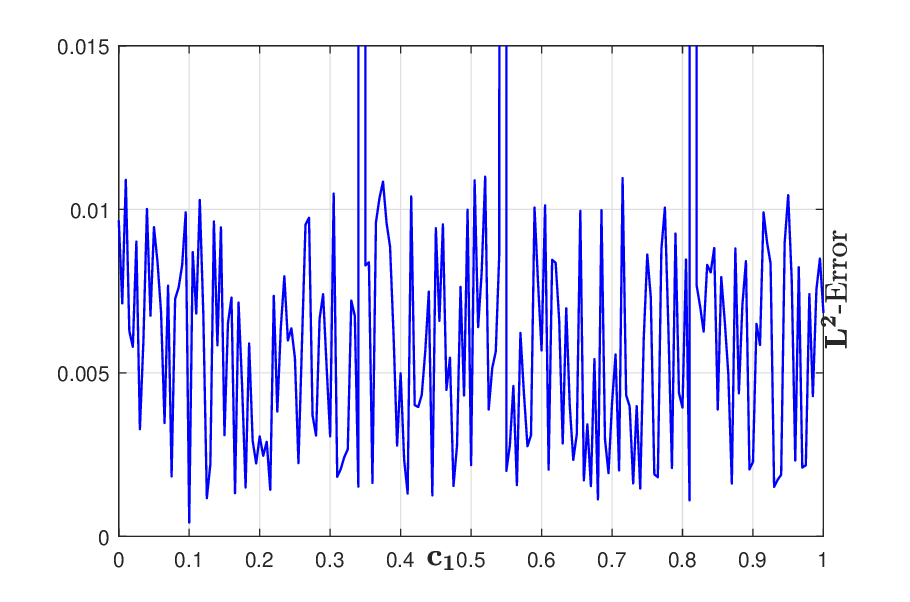}\\
\caption{The evolution of the $L^2$-error in function 
of the first coordinate $c_1$, with $c_2 = 0.7$.}
\label{fig:ex2_evol_x1}
\end{subfigure}
\hfill
\begin{subfigure}[b]{0.49\textwidth}
\centering
\includegraphics[width=1.1\textwidth]{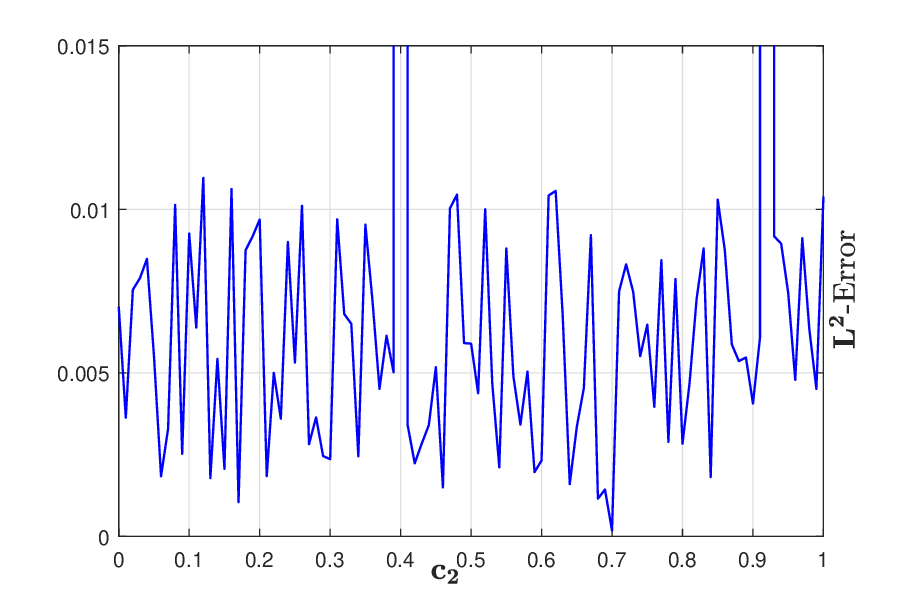}\\
\caption{The evolution of the $L^2$-error in function 
of the second coordinate $c_2$, with $c_1 = 0.1$.}
\label{fig:ex2_evol_x2}
\end{subfigure}
\caption{The evolution of the $L^2$-error versus the location of $c$.}
\label{fig:ex2_evol}
\end{figure}

Firstly, by fixing the value $c_2 = 0.7$ and varying the value of $c_1$ 
over the interval $[0,1]$, we get in the sub-figure \ref{fig:ex2_evol_x1} 
the evolution of the $L^2$-error. After that, we fix $c_1 = 0.1$  
and we vary $c_2$ in order to obtain the sub-figure \ref{fig:ex2_evol_x2}. 
As we can see, the sensor positions $(0.34\ ,\ 0.7),\ (0.55\ ,\ 0.7),\ 
(0.81\ ,\ 0.7),\ (0.1\ ,\ 0.4),$ and $(0.1\ ,\ 0.92)$ are non-strategic 
since the error explodes. Note that there might be other positions of $c$ 
which are also non-strategic, but they were not accounted for by the adopted 
subdivision in the numerical implementation of the example. Again, more 
investigations regarding the theoretical aspects of the relationship 
between the reconstruction error and the sensors' location are needed. 


\section{Conclusion}

The reconstruction problem of the initial state is very important especially in evolution problems 
because it allows us to determine the evolution of the considered system with respect to time. 
We tackled, in this manuscript, the reconstruction problem of the initial state on a boundary 
sub-region $\Sigma$ for a semilinear time-fractional system with the Caputo derivative. 
Under the assumption that the linear part \eqref{sys.lin} is approximately $\omega$-observable, 
we prove that \eqref{sys.semi.lin} is $\mathcal{B}$-observable on $\Sigma$. In order to reconstruct 
the initial state on $\Sigma$, we firstly reconstructed it in $\omega$ by using the results of 
regional observability for semilinear systems; after that we took the restriction on $\Sigma$ 
of its trace on $\partial\omega$ in order to obtain the initial state on $\Sigma$. The condition 
of approximate regional observability of \eqref{sys.lin} is irreplaceable in this paper because 
all the proofs are based on it, but in future works, we are looking to eliminate this condition. 
Many more future directions can be derived from this work. One is the theoretical study regarding 
the relationship between the sensors' location and the $L^2$-Error generated by the proposed algorithm. 
Another one is to investigate the same problem at hand but with a different fractional derivative, 
namely the Atangana--Baleanu or the $\psi$ fractional derivatives, in order to give a comparison 
between the different considered derivatives.    


\subsection*{Author Contributions}

All authors contributed to this work, commented on previous versions of the manuscript, 
read, and approved the final manuscript. K.Z. wrote the original draft preparation; 
K.Z., F.Z.E.A., and D.F.M.T. conceptualized the study;  
K.Z., F.Z.E.A., and D.F.M.T. helped in methodology, 
formal analysis and investigation, writing—review and editing.

\subsection*{Funding}

Torres is supported by CIDMA under the Portuguese Foundation 
for Science and Technology (FCT), Grant UID/04106/2025 
(\url{https://doi.org/10.54499/UID/04106/2025}).

\subsection*{Data availability statement}

No datasets were generated or analyzed during the current study.

\subsection*{Conflict of interest}

The authors declare no conflict of interest.




\begin{thebibliography}{xx}


\bibitem{adams}
Adams, R.A., Fournier, J.J.F.:
\newblock Sobolev Spaces.
\newblock Elsevier (2003).

\bibitem{reg}
Amouroux, M., El Jai, A., Zerrik, E.:
\newblock Regional observability of distributed systems.
\newblock {\em Int. J. Syst. Sci.} 25(2), 301–313 (1994). 
\url{https://doi.org/10.1080/00207729408928961}

\bibitem{ardent}
Arendt, W., Batty, C.J.K., Hieber, M., Neubrander, F.:
\newblock Vector-valued Laplace Transforms and Cauchy Problems.
\newblock Birkhäuser, Basel (2011). 
\url{https://doi.org/10.1007/978-3-0348-0087-7}

\bibitem{app5}
Arshad, S., Baleanu, D., Tang, Y.:
\newblock Fractional differential equations with bio-medical applications.
\newblock In: Handbook of Fractional Calculus with Applications: 
Applications in Engineering, Life and Social Sciences, Part A, 
Berlin, Boston: De Gruyter, pp. 1–20 (2019). 
\url{https://doi.org/10.1515/9783110571905-001}

\bibitem{boun.semi1}
Boutoulout, A., Bourray, H., El Alaoui, F.Z.:
\newblock Regional Boundary Observability 
for Semi-Linear Systems: Approach and Simulation.
\newblock {\em Int. J. Math. Anal.} 4(24), 1153–1173 (2010)

\bibitem{boun.semi2}
Boutoulout, A., Bourray, H., El Alaoui, F.Z.:
\newblock Regional boundary observability of semilinear hyperbolic systems: sectorial approach.
\newblock {\em IMA J. Math. Control Inf.} 32(3), 497–513 (2015). 
\url{https://doi.org/10.1093/imamci/dnu004}

\bibitem{curtain}
Curtain, R.F., Zwart, H.:
\newblock An Introduction to Infinite-Dimensional Linear Systems Theory.
\newblock Springer-Verlag, New York (1995). 
\url{https://doi.org/10.1007/978-1-4612-4224-6}

\bibitem{sup.2}
Dadkhah, E., Shiri, B., Ghaffarzadeh, H., Baleanu, D.:
\newblock Visco-elastic dampers in structural buildings 
and numerical solution with spline collocation methods.
\newblock {\em J. Appl. Math. Comput.} 63, 29–57 (2020). 
\url{https://doi.org/10.1007/s12190-019-01307-5}

\bibitem{sup.1}
Dadkhah Khiabani, E., Ghaffarzadeh, H., Shiri, B., Katebi, J.:
\newblock Spline collocation methods for seismic analysis of 
multiple degree of freedom systems with visco-elastic dampers using fractional models.
\newblock {\em J. Vib. Control} 26(17–18), 1445–1462 (2020). 
\url{https://doi.org/10.1177/1077546319898570}


\bibitem{me.semi.2021}
El Alaoui, F.Z., Boutoulout, A., Zguaid, K.:
\newblock Regional Reconstruction of Semilinear Caputo 
Type Time-Fractional Systems Using the Analytical Approach.
\newblock {\em Adv. Theory Nonlinear Anal. Appl.} 5(4), Art. no. 4 (2021). 
\url{https://doi.org/10.31197/atnaa.799236}

\bibitem{capetact}
El Jai, A.:
\newblock Capteurs et actionneurs dans l’analyse des systèmes distribués.
\newblock Elsevier Masson, Paris (1997)

\bibitem{RegAnal}
Ge, F., Chen Quan, Y., Kou, C.:
\newblock Regional Analysis of Time-Fractional Diffusion Processes.
\newblock Springer, Cham (2018). 
\url{https://doi.org/10.1007/978-3-319-72896-4}

\bibitem{geo}
Henry, D.:
\newblock Geometric Theory of Semilinear Parabolic Equations.
\newblock Springer-Verlag, Berlin Heidelberg (1981). 
\url{https://doi.org/10.1007/BFb0089647}


\bibitem{app3}
Jafari, H., Mehdinejadiani, B., Baleanu, D.:
\newblock Fractional calculus for modeling unconfined groundwater.
\newblock In: Handbook of Fractional Calculus with Applications: 
Applications in Engineering, Life and Social Sciences, 
Part A, Berlin, Boston: De Gruyter, pp. 119–138 (2019). 
\url{https://doi.org/10.1515/9783110571905-007}

\bibitem{kalman}
Kalman, R.E.:
\newblock On the general theory of control systems.
\newblock {\em IFAC Proc. Vol.} 1(1), 491–502 (1960). 
\url{https://doi.org/10.1016/S1474-6670(17)70094-8}

\bibitem{lions.mag}
Lions, J.L., Magenes, E.:
\newblock Non-Homogeneous Boundary Value Problems and Applications, vol. 1.
\newblock Springer-Verlag, Berlin Heidelberg (1972)

\bibitem{app2}
Litak, G., Kwuimy, C.A.K., Ducharne, B.:
\newblock Energy harvesting in dynamical systems with fractional-order physical properties.
\newblock In: Handbook of Fractional Calculus with Applications: Applications 
in Engineering, Life and Social Sciences, Part B, 
Berlin, Boston: De Gruyter, pp. 63–86 (2019). 
\url{https://doi.org/10.1515/9783110571929-003}


\bibitem{app1}
Moghaddam, B.P., Dabiri, A., Machado, J.A.T.:
\newblock Application of variable-order fractional calculus in solid mechanics.
\newblock In: Handbook of Fractional Calculus with Applications: Applications 
in Engineering, Life and Social Sciences, Part A, 
Berlin, Boston: De Gruyter, pp. 207–224 (2019). 
\url{https://doi.org/10.1515/9783110571905-011}

\bibitem{Mu.2017}
Mu, J., Ahmad, B., Huang, S.:
\newblock Existence and regularity of solutions to time-fractional diffusion equations.
\newblock {\em Comput. Math. Appl.} 73(6), 985–996 (2017). 
\url{https://doi.org/10.1016/j.camwa.2016.04.039}

\bibitem{pazy}
Pazy, A.:
\newblock Semigroups of Linear Operators and Applications 
to Partial Differential Equations.
\newblock Springer-Verlag, New York (1983)

\bibitem{appC}
Petráš, I.:
\newblock Handbook of Fractional Calculus with Applications: Applications in Control.
\newblock De Gruyter, Berlin, Boston (2019). 
\url{https://doi.org/10.1515/9783110571745}


\bibitem{app4}
Rossikhin, Y., Shitikova, M.:
\newblock Fractional calculus models in dynamic problems of viscoelasticity.
\newblock In: Handbook of Fractional Calculus with Applications: 
Applications in Engineering, Life and Social Sciences, Part A, 
Berlin, Boston: De Gruyter, pp. 139–158 (2019). 
\url{https://doi.org/10.1515/9783110571905-008}

\bibitem{sup.3}
Shiri, B., Baleanu, D.:
\newblock All linear fractional derivatives with power functions’ 
convolution kernel and interpolation properties.
\newblock {\em Chaos Soliton. Fract.} 170, 113399 (2023). 
\url{https://doi.org/10.1016/j.chaos.2023.113399}

\bibitem{weiss}
Tucsnak, M., Weiss, G.:
\newblock Observation and Control for Operator Semigroups.
\newblock Birkhäuser, Basel (2009). 
\url{https://doi.org/10.1007/978-3-7643-8994-9}



\bibitem{bound1}
Zerrik, E., Bourray, H., Boutoulout, A.:
\newblock Regional boundary observability: A numerical approach.
\newblock {\em Int. J. Appl. Math. Comput. Sci.} 12(2), 143–151 (2002)

\bibitem{bound2}
Zerrik, E., Badraoui, L., El Jai, A.:
\newblock Sensors and regional boundary state reconstruction of parabolic systems.
\newblock {\em Sens. Actuators A Phys.} 75(2), 102–117 (1999). 
\url{https://doi.org/10.1016/S0924-4247(98)00293-3}

\bibitem{me.bound.lin.2022.2}
Zguaid, K., El Alaoui, F.Z.:
\newblock Regional boundary observability for Riemann–Liouville linear fractional evolution systems.
\newblock {\em Math. Comput. Simul.} 199, 272–286 (2022). 
\url{https://doi.org/10.1016/j.matcom.2022.03.023}

\bibitem{me.semi.lin.2023}
Zguaid, K., El Alaoui, F.Z.:
\newblock Regional Boundary Observability for Semilinear 
Fractional Systems with Riemann-Liouville Derivative.
\newblock {\em Numer. Funct. Anal. Optim.} 44(5), 420–437 (2023). 
\url{https://doi.org/10.1080/01630563.2023.2171055}

\bibitem{me.2023.chap}
Zguaid, K., El Alaoui, F.Z.:
\newblock The Regional Observability Problem for a Class of Semilinear 
Time-Fractional Systems With Riemann-Liouville Derivative.
\newblock In: Debnath, P., Torres, D.F.M., Cho, Y.J. (eds.) Advanced Mathematical 
Analysis and its Applications, pp. 251–264. CRC Press, Boca Raton (2023). 
\url{https://doi.org/10.1201/9781003388678-15}

\bibitem{me.2023.ajc}
Zguaid, K., El Alaoui, F.Z., Boutoulout, A.:
\newblock Regional observability of Caputo semilinear fractional systems.
\newblock {\em Asian J. Control} (2023). 
\url{https://doi.org/10.1002/asjc.3218}

\bibitem{me.2023.ijdyc}
Zguaid, K., El Alaoui, F.Z., Torres, D.F.M.:
\newblock Regional gradient observability for fractional 
differential equations with Caputo time-fractional derivatives.
\newblock {\em Int. J. Dyn. Control} 11(5), 2423–2437 (2023). 
\url{https://doi.org/10.1007/s40435-022-01106-0}

\end{thebibliography}
\end{document}